\documentclass[12pt,letterpaper]{amsart}
\usepackage{amsmath,amsfonts,amssymb, tikz, amsthm}
\usepackage[colorlinks=true,linkcolor=black,anchorcolor=black,citecolor=black,filecolor=black,menucolor=black,runcolor=black,urlcolor=blue]{hyperref}
\usepackage[utf8]{inputenc}
\usepackage[english]{babel}
\usepackage{xcolor}
\usepackage[tableposition=top]{caption} 
\usepackage{epigraph}
\usepackage{mathtools}

\newtheorem*{theorem*}{Theorem}
\newtheorem{theorem}{Theorem}
\newtheorem{lemma}{Lemma}

\newtheorem{corollary}{Corollary}

\theoremstyle{remark}

\newtheorem{remark}{Remark}

\numberwithin{equation}{section}
\numberwithin{corollary}{theorem}
\numberwithin{remark}{section}
\numberwithin{assumption}{section}

\Alph{assumption}

\newcommand{\Z}{\mathbb{Z}}
\newcommand{\Q}{\mathbb{Q}}
\newcommand{\C}{\mathbb{C}}

\newcommand{\N}{\mathbb{N}}
\newcommand{\F}{\mathbb{F}}

\begin{document}
\title[On some symmetries of the base $ n $ expansion of $ 1/m $ - 2]{On some symmetries of the base $ n $ expansion of $ 1/m $ : The Class Number connection}

\author{Kalyan Chakraborty}
\email[Kalyan Chakraborty]{kalychak@ksom.res.in}
\address{ Kerala School of Mathematics, KCSTE,
	Kunnamangalam, Kozhikode, Kerala, 673571, India.}

\author{Krishnarjun Krishnamoorthy}
\email[Krishnarjun K]{krishnarjunk@hri.res.in, krishnarjunmaths@gmail.com}
\address{Harish-Chandra Research Institute,  A CI of Homi Bhabha National
	Institute, Chhatnag Road, Jhunsi, Prayagraj - 211019, India.}

\maketitle

\begin{abstract}
	Suppose that $ m\equiv 1\mod 4 $ is a prime and that $ n\equiv 3\mod 4 $ is a primitive root modulo $ m $. In this paper we obtain a relation between the class number of the imaginary quadratic field $ \Q(\sqrt{-nm}) $ and the digits of the base $ n $ expansion of $ 1/m $.
	
	Secondly, if $ m\equiv 3\mod 4 $, we study some convoluted sums involving the base $ n $ digits of $ 1/m $ and arrive at certain congruence relations involving the class number of $ \Q(\sqrt{-m}) $ modulo certain primes $ p $ which properly divide $ n+1 $.
\end{abstract}

\section{Introduction}\label{Section "Introduction"}

Class numbers are mysterious objects which occur naturally in number theory. We shall quickly glance through landmark results which are relevant to the present work and refer the reader to \cite{Bha-Mur} and the references therein for further information.

Suppose that $ D>0 $ is square free and suppose we denote the class number of $ \Q(\sqrt{-D}) $ as $ h(D) $. Concerning the size of $ h(D) $, it is known that \cite{Siegel}, for every $ \epsilon >0 $ there exists constants $ c_1,c_2 $ such that 
\begin{equation}\label{Equation "Class number Growth"}
	c_1 D^{1/2-\epsilon} < h(D) <c_2 D^{1/2+\epsilon}.
\end{equation}

The next object of attention is the divisibility properties of $ h(D) $. Below we mention two results, one regarding non divisibility \cite{Kohnen-Ono} and another regarding divisibility \cite{Sound - divisibility} of class numbers.

\begin{theorem*}[Kohnen-Ono]\label{Theorem "Kohnen-Ono"}
	Suppose that $ \ell >3 $ is a prime, then for every $ \epsilon>0 $ and large enough $ X $, we have 
	\[
	\#\{0 < D < X\ |\ \ell\nmid h(D)\} \geq \left(\frac{2(\ell-2)}{\sqrt{3}(\ell-1)} - \epsilon\right) \frac{\sqrt{X}}{\log(X)}.
	\]
\end{theorem*}

\begin{theorem*}[Soundararajan]\label{Theorem "Sound"}
	Suppose that $ \ell >3 $ is a prime, then for large enough $ X $, we have 
	\[
	\#\{0 < D < X\ |\ \ell |h(D)\} \gg X^{1/2}.
	\]
\end{theorem*}

Now suppose that $ m\equiv 3\mod 4 $ is a prime and that $ n $ is a primitive root modulo $ m $. This means that $ n $ generates $ (\Z/m\Z)^\times $. Then the class number $ h(m) $ of $\Q(\sqrt{-m})$ has been connected to the base $ n $ expansion of $ 1/m $, \cite{Girstmair 1}, \cite{Girstmair 2}, \cite{Girstmair 3}.
\begin{theorem}[Girstmair]\label{Theorem "Girstmair"}
	If $ m\equiv 3\mod 4 $ is a prime and $ 1 < n < m $ is a primitive root modulo $ m $, then 
	\[
	\sum_{k=1}^{m-1} (-1)^k a_k(m) = (n+1)h(m),
	\]
	where $ a_k(m) $ are the digits of the base $ n $ expansion of $ 1/m $.
\end{theorem}

\begin{theorem}[Girstmair]\label{Theorem "Ram-Thanga"}
	If $ m\equiv 3\mod 4 $ is a prime and $ n $ be co prime to $ m $ with order $ (m-1)/2 $ in $ (\Z/m\Z)^\times $. Then,
	\[
	\left(\frac{n-1}{2}\right) \left(\frac{m-1}{2} - h(m)\right) = \sum_{k=1}^{(m-1)/2} a_k(m).
	\]
\end{theorem}
The study was furthered by Murty and Thangadurai \cite{Ram-Tha}.

The first purpose of this paper is to prove the following theorem which is an analogue of Theorems \ref{Theorem "Girstmair"}, \ref{Theorem "Ram-Thanga"} for composite discriminants. The relation that we arrive at is slightly complicated but is in the spirit of providing an expression for the class number of quadratic fields in terms of digits in some base $ n $. Define for every $ 0\leq a\leq n-1 $
\begin{align}
	\sigma(a,1) &:= \#\{1\leq k\leq (m-1)\ |\ ma_k(m)\equiv a\mod n\mbox{ and } 2|k\},\label{Equation "sigma 1"}\\
	\sigma(a,-1) &:= \#\{1\leq k\leq (m-1)\ |\ ma_k(m)\equiv a\mod n\mbox{ and } 2\nmid k\},\label{Equation "sigma 2"}
\end{align}
where $ a_k(m) $'s are the digits of the base $ n $ expansion of $ 1/m $. With these definitions we state our first main theorem.
\begin{theorem}\label{Theorem "Class Number main"}
	Let $ m\equiv 1\mod 4 $ be a prime and $ n\equiv 3\mod 4 $ be a square free integer which is a primitive root modulo $ m $. Observe that in particular $ n $ is coprime to $ m $. Then
	\[
	h(nm) = \frac{-1}{n} \left(\sum_{k=1}^{n-1} k \sum_{a=1}^{n} \chi_1\left(a + k m\right)\left(\sigma(-a,1) - \sigma(-a,-1)\right)\right),
	\]
	where $ h(nm) $ is the class number of $ \Q(\sqrt{-nm}) $ and $ \chi_1 $ is the Legendre symbol modulo $ n $.
\end{theorem}

In particular suppose that $ n=3 $. Since we assume that $ 3 $ is a primitive root modulo $ m $ it follows that $ 3 $ is a quadratic non-residue modulo $ m $ and therefore by quadratic reciprocity, we see that $ m\equiv 2\mod 3 $. In particular if we suppose that $ m=3k+2 $ for some integer $ k $ we have the following corollary.
\begin{corollary}\label{Corollary "Class Number corollary"}
	Suppose $ n=3 $ with the rest of the notations as in Theorem \ref{Theorem "Class Number main"}. Then 
	\begin{equation}\label{Equation "Corollary 1"}
		h(3m) = 2k - 4\sigma(0,1).
	\end{equation}
\end{corollary}
In Table \ref{Table 1} below, we illustrate Corollary \ref{Corollary "Class Number corollary"} by computing a few examples for the first few values of $ m $. The corresponding class numbers are obtained from the LMFDB database\footnote{\url{https://www.lmfdb.org/NumberField/}}.

\begin{table}[h!]
	\centering
	\caption{Some examples}
	\label{Table 1}
	\begin{tabular}{ |c|c|c|c|c| } 
		\hline
		$ m $ & $ k $ & $ \sigma(0,1) $ & $ 2k - 4\sigma(0,1) $ & $ h(3m) $\\
		\hline
		 5 & 1 & 0 & 2 & 2 \\ 
		 17 & 5  & 2 & 2  & 2\\
		 29 & 9  & 3 & 6  & 6\\
		 101 & 33  & 14 & 10  & 10\\
		 113 & 37  & 17 & 6  & 6\\
		 197 & 65  & 27 & 22  & 22\\
		 269 & 89  & 41 & 14  & 14\\
		 317 & 105  & 46 & 26  & 26\\
		 557 & 185  & 83 & 38  & 38\\
		 653 & 217  & 98 & 42  & 42\\
		\hline
	\end{tabular}
\end{table}

Using Scholz's reflection theorem (see Theorem 10.10 of \cite{Larry W- Cyclotomic fields}) we have the following corollary.
\begin{corollary}\label{Corollary "Class Number corollary 2"}
	Suppose we denote by $ h_>(m) $ the class number of $ \Q(\sqrt{m}) $, if $ k\not\equiv 2\sigma(0,1) \mod 3 $, then $ 3\nmid h_>(m) $.
\end{corollary}
As we shall see below, we have the following equivalent definition for $ \sigma(0,1) $,
\begin{equation}\label{Equation "Sigma alternate"}
	\sigma(0,1) = \#\left\{1\leq l\leq m\ |\ 3|l\mbox{ and } \chi_2(l)=1\right\},
\end{equation}
where $ \chi_2 $ is the Legendre symbol modulo $ m $. Furthermore, we can also deduce the growth of $ \sigma(0,1) $ from \eqref{Equation "Class number Growth"} and \eqref{Equation "Corollary 1"}. Intuitively from \eqref{Equation "Sigma alternate"} we feel that $ \sigma(0,1) $ should grow like $ m/6 $. Indeed this is the case. We state this result as an easily proved corollary.

\begin{corollary}\label{Corollary "Class Number corollary 3"}
	With notations as above, as $ m\to\infty $ we have
	\[
	\sigma(0,1) = \frac{m}{6} + O_\epsilon(m^{1/2+\epsilon}).
	\]
\end{corollary}
We can interpret Corollary \ref{Corollary "Class Number corollary 3"} as the independence of the events ``$ l $ is divisible by $ 3 $" and ``$ l $ is a quadratic residue modulo $ m $" (see also \S \ref{Section "Concluding Remarks"}). We can deduce corresponding corollaries for $ \sigma(0,-1) $ also.

A striking difference between Theorems \ref{Theorem "Girstmair"}, Theorem \ref{Theorem "Ram-Thanga"} and Theorem \ref{Theorem "Class Number main"} warrants a comment. The sums considered in Theorems \ref{Theorem "Girstmair"} and \ref{Theorem "Ram-Thanga"} involve the digits $ a_k(m) $ themselves, whereas in Theorem \ref{Theorem "Class Number main"} involves only the number of digits of certain type. Nevertheless in the investigations leading up to the proof of Theorem \ref{Theorem "Class Number main"}, we are naturally led to consider certain other sums involving the digits of $ 1/m $ modulo some base $ n $. 

Suppose we let $ \mathcal{O}_m(n) $ denote the order of $ n $ in the subgroup $ (\Z/m\Z)^\times $. We consider sums of the form 
\begin{equation}\label{Equation "Definition S"}
	S := \sum_{k=1}^{\mathcal{O}_m(n)} a_k(m) n^{\mathcal{O}_m(n)-k}k,
\end{equation}
which can be thought of as weighted versions of the sums considered in Theorems \ref{Theorem "Girstmair"} and \ref{Theorem "Ram-Thanga"}. We will be interested in considering $ S $ when $ m \equiv 3 \mod 4 $ is a prime and $ n $ is a primitive root modulo $ m $ even though we have defined $ S $ in slighter generality. The study of these sums is interesting in their own right and have led to the second main result of this paper. 

\begin{theorem}\label{Theorem "Congruence for S"}
	Suppose that $ m\equiv 3\mod 4 $ is a prime, $ S $ is as above and $ p $ is an odd prime. The following results hold.
	\begin{enumerate}
		\item	If $ n $ is a primitive root modulo $ m $ and $ p\|(n+1) $, then 
		\begin{equation}\label{Equation "mS congruence 1"}
			mS \equiv 1-m-mh(m) \mod p.
		\end{equation}
		\item	If $ n $ is a primitive root modulo $ m $ and $ p|(n-1) $, then
		\begin{equation}\label{Equation "mS congruence 2"}
			mS \equiv 1 - m + mq \mod p,
		\end{equation}
	\end{enumerate}
	where we have set $ q:=(m-1)/2 $ and $ h(m)$ as usual is the class number of $ \Q(\sqrt{-m}) $.
\end{theorem}

An interesting observation is that in \eqref{Equation "mS congruence 1"} and \eqref{Equation "mS congruence 2"}, every term except $ S $ is independent of choice of $ n $. This tells us that modulo a fixed prime $ p $ say, the sum $ S $ independent of choice of $ n $ which satisfy the hypotheses of the theorem. For example if $ l $ is an odd number coprime to $ p $ such that $ k\mapsto k^l $ is an automorphism of $ (\Z/m\Z)^\times $ and if $ n^l < m $, then the theorem can be applied for $ n^l $ in place of $ n $. This study may lead to insights into additional structure of $ S $.

Even though the two main results seem to be quite closely related to one another, the proofs are rather independent of each other. Therefore we shall prove Theorem \ref{Theorem "Congruence for S"} before proving Theorem \ref{Theorem "Class Number main"} and the corollaries. In \S \ref{Section "Preliminary Results"} we collect some standard results regarding the digits $ a_k(m) $ and fix the notations along the way. We shall turn to the proof of Theorem \ref{Theorem "Congruence for S"} in \S \ref{Section "Distribution of prime moduli"} and the proof of Theorem \ref{Theorem "Class Number main"} in \S \ref{Section "Class number pq"}. Finally in \S \ref{Section "Concluding Remarks"} we discuss some interesting questions and hint at further directions which can be taken.

We conclude the introduction with the following result. The function field analogue of the digit sums considered above is rather easily solved. In a personal communication, Prof. Z\'eev Rudnick\footnote{The Cissie and Aaron Beare Chair in Number Theory, Department of Mathematics, Tel-Aviv University.} has informed the authors of the following result. 

Let $ q $ be a prime power and $ \F_q[x], \F_q(x) $ denote the polynomial ring and the field of rational functions in one variable with coefficients in $ \F_q $ respectively. Fix $ B(x)\in \F_q[x] $ of positive degree. Then any $ f\in \F_q (x) $ has a base $ B $ expansion as follows,
\[
f(x) = \sum_{j=-\infty}^{J} a_{-j}(x) B^j(x),
\]
where each $ a_i(x) \in \F_q[x]$ with degree strictly less than that of $ B(x) $. In particular if we take $ P\in\F_q[x] $ to be an irreducible polynomial coprime to $ B $
and consider $ f=1/P $, then as in the case of integers, it follows that the base $ B $ expansion of $ 1/P $ is periodic with period say $ T $. In this setup, in analogy with the integers we can define the following sum,
\[
S_B(P)(x) = a_1(x) + a_2(x) + \ldots + a_T(x) \in \F_q[x].
\]
Now we have the following theorem.

\begin{theorem}[Rudnick]\label{Theorem "Rudnick"}
	If $ P $ is coprime to $ B(B-1) $ then $ S_B(P)(x) = 0$.
\end{theorem}
The proof of the above theorem is rather straightforward and hence is omitted. However, Theorem \ref{Theorem "Rudnick"} highlights that the behaviour of the digits described in Theorems \ref{Theorem "Girstmair"} and \ref{Theorem "Ram-Thanga"} are somewhat unique to the integer setup. It stands to reason that the underlying role played by arithmetic is much deeper than what one might initially perceive.

\section*{Acknowledgements}

The second named author would like to thank Sreejith M.M. and Mohit Mishra for useful conversations and the Kerala School of Mathematics for its generous hospitality. The authors also would like to thank Zeev Rudnick for sharing with us the proof of Theorem \ref{Theorem "Rudnick"} and for his other valuable comments on the manuscript. The authors would also like to thank the referee for their valuable comments.

\section{Preliminary Results}\label{Section "Preliminary Results"}

\subsection*{Notation}

Consider an integer $ n\geq 2 $ which shall remain fixed for the rest of the section. Suppose that $ m $ is another natural number such that $ (n,m)=1 $. Consider the base $ n $ representation of $ 1/m $,
\begin{equation}\label{Equation "Base n representation 1/m"}
	\frac{1}{m} = \sum_{k=1}^{\infty} \frac{a_k(m)}{n^k}.
\end{equation}
The coefficients (which we shall refer to as digits from now on) $ a_k(m) $ each satisfy $ 0\leq a_k(m)<n $ and are unique given $ m $. Furthermore they are periodic in $ k $ and the period equals the order of $ n $ in $ (\Z/m\Z)^\times $. 

Suppose that $ 0\leq a_1,\ldots,a_l < n $ are integers. The ordered tuple $ (a_1,\ldots,a_l) $ is then called a \textit{string of length} $ l $ modulo $ n $. Given a string $ (a_1,\ldots,a_l) $ modulo $ n $, define
\begin{align*}
	\overline{(a_1,\ldots,a_l)} & := \sum_{k=1}^{\infty} \frac{a_k}{n^k},
\end{align*}
where, for $ k >l $ we take $ a_k := a_{k-l} $. We shall denote the order of $ n $ in $ (\Z/m\Z)^\times $ as $ \mathcal{O}_m(n) $. Furthermore, the class number of $ \Q(\sqrt{-m}) $ will be denoted by $ h(m) $. We shall assume that $ m\equiv 3\mod 4 $ unless otherwise mentioned and set $ q=(m-1)/2 $.

%

%

\begin{lemma}\label{Lemma 2}
	Suppose that $ 1/m = \overline{(a_1,\ldots,a_l)} $. If $ a\equiv n^t \mod m $ for some $ t\geq 0 $, then, 
	\[
	\left\{\frac{a}{m}\right\} = \overline{(a_{t+1},\ldots, a_l,a_1,\ldots, a_t)}.
	\]
\end{lemma}

\begin{proof}
	The proof follows from the observation that
	\[
	\left\{\frac{a}{m}\right\} = \left\{\frac{n^t}{m}\right\}.
	\]
\end{proof}

%
%

\begin{lemma}
	For any $ k $, we have 
	\[
	a_k(m) \equiv \left\lfloor\frac{n^k}{m}\right\rfloor  \mod n.
	\]
	Furthermore, if $ p $ is a prime such that $ n\equiv 1\mod p $, then for any $ d $,
	\[
	\sum_{k=1}^d a_k(m) \equiv \left\lfloor\frac{n^d}{m}\right\rfloor \mod p
	\]
\end{lemma}

\begin{proof}
	The proof is clear since 
	\begin{equation}\label{Equation "ak(m) formula"}
		\left\lfloor\frac{n^d}{m}\right\rfloor = \sum_{k=1}^d a_k(m) n^{d-k}.
	\end{equation}
\end{proof}
Suppose for an integer $ g $ we denote by $ [g]_n $ to be the least \textit{non-negative} integer such that $ [g]_n\equiv g\mod n $.
\begin{lemma}\label{Lemma 3}
	With $ n,m $ as above, we have 
	\[
	a_k(m) = \frac{n[n^{k-1}]_m-[n^k]_m}{m}.
	\]
\end{lemma}

\begin{proof}
	We have,
	\begin{align*}
		\sum_{k=1}^\infty \frac{n[n^{k-1}]_m-[n^k]_m}{mn^k}&= \frac{1}{m} \sum_{k=1}^\infty\frac{n[n^{k-1}]_m-[n^k]_m}{n^k}\\
		&= \frac{1}{m}\left(	\sum_{k=1}^\infty \frac{[n^{k-1}]_m}{n^{k-1}} - \sum_{k=1}^\infty\frac{[n^k]_m}{n^k}\right)\\
		&=\frac{1}{m}.
	\end{align*}
Now the lemma follows from the uniqueness of $ a_k(m) $s once we observe that $ \frac{n[n^{k-1}]_m-[n^k]_m}{m} $ are integers lying in between $ 0 $ and $ n-1 $.
\end{proof}

As a particular consequence of the previous lemma, we have that
\begin{equation}\label{Equation "Congruence of ak"}
	ma_k(m)\equiv -[n^k]_m \mod n.
\end{equation}

Another consequence of the above lemma is the following lemma.

\begin{lemma}\label{Lemma 5}
	Suppose that $ m $ is a prime and that $ \mathcal{O}_m(n) $ is even for some $ n\not\equiv -1\mod m $. Then for every $ k $ we have 
	\[
	a_k(m) + a_{k+\frac{\mathcal{O}_m(n)}{2}}(m) = n-1.
	\]
\end{lemma}

\begin{proof}
	From \eqref{Equation "Congruence of ak"} we have
	\[
	m(a_k(m)+a_{k+\frac{\mathcal{O}_m(n)}{2}}(m)) \equiv -[n^k]_m - [n^{k+\frac{\mathcal{O}_m(n)}{2}}]_m\mod n.
	\]
	Since $ n^{k+\frac{\mathcal{O}_m(n)-1}{2}}\equiv -n^k\mod m $, we have 
	\begin{align*}
		m(a_k(m)+a_{k+\frac{\mathcal{O}_m(n)}{2}}(m)) &\equiv -[n^k]_m + [n^k]_m -m \mod n\\
		&\equiv -m \mod n.
	\end{align*}
	Since $ (n,m)=1 $, this in particular gives us that 
	\[
	a_k(m)+a_{k+\frac{\mathcal{O}_m(n)}{2}}(m)\equiv -1\mod n.
	\]
	Now since $ 0\leq a_k(m),a_{k+\frac{\mathcal{O}_m(n)}{2}}(m) \leq n-1 $, the claim follows.
\end{proof}

\begin{remark}
	It would be extremely interesting to find an analogue of Lemma \ref{Lemma 5} when $ \mathcal{O}_m(n) $ is odd.
\end{remark}

\section{Class Numbers of Prime Discriminants}\label{Section "Distribution of prime moduli"}

We shall use Theorems \ref{Theorem "Girstmair"} and \ref{Theorem "Ram-Thanga"} freely in what follows. Suppose that $ n < m $ and that $ p $ is a prime other than $ m $.

\begin{proof}[Proof of Theorem \ref{Theorem "Congruence for S"}]
	\begin{enumerate}
		\item	Observe that following equation holds generally, 
		\begin{equation}\label{Equation "Temp"}
			\sum_{k=1}^{m-1} a_k(m) n^{m-k} = n\left\lfloor \frac{n^{m-1}}{m}\right\rfloor= \frac{n^{m}-n}{m}.
		\end{equation}
		Suppose that $ n=lp-1 $ and set $ q=(m-1)/2 $. Since $ m\equiv 3 \mod 4$ is a prime and $ n $ is a primitive root modulo $ m $, we have
		\begin{align*}
			\frac{n^m-n}{m} &=\sum_{k=1}^{m-1} a_k(m) n^{m-k}\\
			&=\sum_{k=1}^{m-1} a_k(m) (lp-1)^{m-k}\\
			&\equiv \sum_{k=1}^{m-1} (-1)^{m-k-1} a_k(m) (m-k)lp + \sum_{k=1}^{m-1} (-1)^{m-k} a_k(m) \\
			&\equiv mlp\sum_{k=1}^{m-1} (-1)^{m-k-1} a_k(m) + lp\sum_{k=1}^{m-1} (-1)^{m-k} a_k(m)k + \sum_{k=1}^{m-1} (-1)^{m-k} a_k(m)\\
			&\equiv (1-mlp)\sum_{k=1}^{m-1} (-1)^{m-k} a_k(m) + lp\sum_{k=1}^{m-1} (-1)^{m-k} a_k(m)k\\
			&\equiv -lph(m) - lp\sum_{k=1}^{m-1} (-1)^{k} a_k(m)k\\
			&\equiv -lp\left(h(m) + \sum_{k=1}^{m-1} (-1)^{k} a_k(m)k \right)\mod p^2.
		\end{align*}		
		Now consider $ n^{m-1}-1 $. Since $ p \| (n^2-1)$, we have $ p^2 | (n^{m-1}-1) $ if and only if $ p | (n^{m-1}-1)/(n^2-1) $. Suppose that $ m-1=2q $ for some odd $ q $, then we have
		\begin{align*}
			\frac{n^{m-1}-1}{n^2-1} &= \left((n^2)^{q-1} + (n^2)^{q-2} + \ldots + n^2 + 1 \right)\\
			&\equiv q\mod p,
		\end{align*}
		since $ n\equiv -1\mod p $. In particular we have 
		\[
		-2nlpq \equiv -lpm\left(h(m) + S\right) \mod p^2.
		\]
		From assumption, we have $ (l,p)=1 $. This finally gives us
		\begin{equation}\label{Equaiton "Congruence 1"}
			2q \equiv m-1 \equiv -m(h(m) + S) \mod p,
		\end{equation}
	from where the required congruence follows.
	\item	Recall that from our assumptions, $ n\equiv 1\mod p $. Also from \eqref{Equation "Temp"}, 
	\[
	\sum_{k=1}^{m-1} a_k(m)m \equiv 0\mod p.
	\]
	We start with the following expression,
	\begin{align*}
		\sum_{k=1}^{m-1} \left\lfloor \frac{n^k}{m}\right\rfloor &\equiv \sum_{k=1}^{m-1}\sum_{d=1}^{k} a_d(m)\\
		&\equiv \sum_{k=1}^{m-1} a_k(m) (m-k)\\
		&\equiv -S\mod p.
	\end{align*}
On the other hand,
	\begin{align*}
		m\sum_{k=1}^{m-1} \left\lfloor \frac{n^k}{m}\right\rfloor &= \sum_{k=1}^{m-1} n^k - \sum_{k=1}^{m-1} [n^k]_m\\
		&\equiv (m-1) -mq \mod p.
	\end{align*}
The required result follows from the previous two congruences.
	\end{enumerate}
\end{proof}

\section{The Class Numbers of Squarefree Discriminants}\label{Section "Class number pq"}

\subsection{Some General Character sums}\label{Section "Bernoulli"}

For the remaining of this section fix $ D\in \N $ and $ \chi $ an \textit{odd} Dirichlet character modulo $ D $ such that $ \chi $ can be written as a product of Dirichlet characters $ \chi_1,\chi_2 $ modulo integers $ n,m >1 $. Suppose that $ \chi_1 $ is odd and $ \chi_2 $ is even, $ D=nm $ and $ (n,m) = 1 $. In particular we note that,
\begin{align}\label{Equation "Character Vanish"}
	\sum_{a=1}^{n} \chi_1(a) = 0&&\mbox{and}&& \sum_{a=1}^{m} \chi_2(a) a = 0.
\end{align}

The generalized Bernoulli numbers (see Chapter 2 of \cite{Iwa}) are defined as the coefficients of 
\begin{equation}\label{Equation "Generalized Bernoulli numbers definition"}
	F_\chi = \sum_{a=1}^{D} \frac{\chi(a) t e^{at}}{e^{Dt}-1} = : \sum_{n=0}^{\infty} B_n(\chi) \frac{t^n}{n!}.
\end{equation}
In particular (see pg. 14 of \cite{Iwa}), we have
\[
B_1(\chi) = \frac{1}{D} \sum_{a=1}^{D} \chi(a) a,
\]
or equivalently, we have
\begin{equation}\label{Equation "B1 formula"}
	DB_1(\chi) = \sum_{a=1}^{D} \chi(a) a.
\end{equation}
In particular, if $ D < -4 $ is a fundamental discriminant and $ \chi $ is the Legendre symbol modulo $ D $, we have
\begin{equation}\label{Equation "Class Number - Bernoulli"}
	B_1(\chi) = -h(D).
\end{equation}
Let $ I_1, I_2\subset \C $ be the images of $ \chi_1,\chi_2 $. 
Define, for an integer $ k $ and $ \alpha\in I_2 $
\begin{equation}\label{Equation "Sk alpha definition"}
	\beta(a,\alpha) : = \#\{1\leq l\leq m\ |\ l\equiv a\mod n, \chi_2(l)=\alpha\}.
\end{equation}

Now we can state the main theorem of this section.
\begin{theorem}\label{Theorem "General character sums"}
	With $ D, n, m $ as above and $ m $ a prime, we have 
	\[
	B_1(\chi) = \frac{1}{n}\sum_{\alpha\in I_2} \alpha \sum_{k=0}^{n-1} k \sum_{a=1}^{n} \chi_1(a+km)\beta(a,\alpha).
	\]
\end{theorem}
\begin{proof}
We begin by considering the similar expression for $ B_1(\chi_1) $.
\begin{align*}
	\sum_{a=1}^{nm} \chi_1(a) a &= \sum_{k=0}^{m-1}\sum_{a=1}^{n} \chi_1(a+kn) (a+kn)\\
	&=\sum_{k=0}^{m-1}\sum_{a=1}^{n} \chi_1(a+kn) a + n \sum_{k=0}^{m-1}k\sum_{a=1}^{n} \chi_1(a+kn).
\end{align*}

The second term in the previous summation vanishes because of the perodicity of $ \chi_1 $ and \eqref{Equation "Character Vanish"}. Similarly, due to the periodicity of $ \chi_1 $ and \eqref{Equation "B1 formula"} the first term equals $ nmB_1(\chi_1) $. Therefore we have shown that,
\begin{equation}\label{Equation "D1D2B1 - 1"}
	nmB_1(\chi_1) = \sum_{a=1}^{nm} \chi_1(a) a.
\end{equation}

Furthermore from \eqref{Equation "B1 formula"} we have,
\begin{align*}
	nm\left(B_1(\chi) - B_1(\chi_1)\right) &= \left(\sum_{a=1}^{nm}\chi(a) a- \sum_{a=1}^{nm} \chi_1(a) a\right)\\
	&= \sum_{a=1}^{nm}\chi_1(a)a \left(\chi_2(a) - 1\right).
\end{align*}
In terms of $ I_1,I_2 $, this can be rewritten as
\begin{equation*}
	\sum_{a=1}^{nm}\chi_1(a)a \left(\chi_2(a) - 1\right) = \sum_{\alpha\in I_2} \left(\alpha - 1\right) \sum_{\underset{\chi_2(a) = \alpha}{a=1}}^{nm}\chi_1(a)a.
\end{equation*}
For every $ \alpha\in I_2 $, set 
\begin{equation}\label{Equation "Gamma alpha definition"}
	\Gamma_\alpha = \sum_{\underset{\chi_2(a) = \alpha}{a=1}}^{nm} \chi_1(a)a.
\end{equation}
In particular, if $ \alpha=0 $, then we have
\begin{align*}
	\Gamma_0 &= \sum_{k=1}^{n} \chi_1(km)km\\
	&= m \chi_1(m)\sum_{k=1}^{n}\chi_1(k)k\\
	&= \chi_1(m)nmB_1(\chi_1).
\end{align*}
Plugging everything together we get,
\begin{equation}\label{Equation "Bernoulli intermediate - "}
	nm\left(B_1(\chi) - \left(1-\chi_1(m)\right)B_1(\chi_1)\right) = \sum_{\alpha\in I_2\setminus\{0\}} (\alpha-1) \Gamma_\alpha.
\end{equation}
Similarly we have,
\begin{equation}\label{Equation "Bernoulli intermediate + "}
	nm\left(B_1(\chi) + \left(1 - \chi_1(m)\right)B_1(\chi_1)\right) = \sum_{\alpha\in I_2\setminus\{0\}} (\alpha+1) \Gamma_\alpha.
\end{equation}

Now, fix $ \alpha\in I_2\setminus\{0\} $ and consider $ \Gamma_\alpha $. We have,
\begin{align*}
	\Gamma_\alpha &= \sum_{\underset{\chi_2(a) = \alpha}{a=1}}^{nm} \chi_1(a)a\\
	&= \sum_{k=0}^{n-1} \sum_{\underset{\chi_2(a) = \alpha}{a=1}}^{m} \chi_1(a+km)(a+km)\\
	&= \sum_{k=0}^{n-1} \sum_{\underset{\chi_2(a) = \alpha}{a=1}}^{m} \chi_1(a+km)a + m\sum_{k=0}^{n-1} \sum_{\underset{\chi_2(a) = \alpha}{a=1}}^{m} \chi_1(a+km)k\\
	&= \sum_{\underset{\chi_2(a) = \alpha}{a=1}}^{m} a \sum_{k=0}^{n-1} \chi_1(a+km) + m\sum_{k=0}^{n-1} \sum_{\underset{\chi_2(a) = \alpha}{a=1}}^{m} \chi_1(a+km)k.
\end{align*}
In the first summation, as $ (n,m)=1 $, the map $ a\mapsto a+km $ is a bijection on the set $ \Z/n\Z $ and therefore the first sum vanishes for every $ a $. This gives us
\begin{align*}
	\Gamma_\alpha &= m\sum_{k=0}^{n-1} \sum_{\underset{\chi_2(a) = \alpha}{a=1}}^{m} \chi_1(a+km)k\\
	&= m\sum_{k=0}^{n-1} k \sum_{\underset{\chi_2(a) = \alpha}{a=1}}^{m} \chi_1(a+km).
\end{align*}
For an integer $ k $, let 
\begin{equation}\label{Equation "ck definition"}
	c_k (\alpha) = \sum_{\underset{\chi_2(a) = \alpha}{a=1}}^{m} \chi_1(a+km).
\end{equation}
It is clear that $ c_k(\alpha) = c_{k+n}(\alpha) $. Now from \eqref{Equation "Sk alpha definition"} we see that,
\begin{align*}\nonumber
	c_k(\alpha) &= \sum_{a=1}^{n} \chi_1(a)(\#\{km+1\leq b\leq (k+1)m\ |\  b\equiv a \mod (n), \chi_2(b)=\alpha\})\\
	&= \sum_{a=1}^{n} \chi_1(a)(\#\{1\leq b\leq m\ |\  b\equiv a - km \mod (n), \chi_2(b)=\alpha\})\\
	&= \sum_{a=1}^{n} \chi_1(a)\beta (a-km,\alpha).
\end{align*}
Since $ \beta(a,\alpha) $ is periodic in $ a $ with period $ n $, the expression on the right hand side of the above chain of equalities is invariant under the map $ a\mapsto a+1 $, since we are summing over all cosets modulo $ n $.
Hence we get
\begin{equation}\label{Equation "ck formula"}
	c_k(\alpha) = \sum_{a=1}^{n} \chi_1(a+km)\beta (a,\alpha).
\end{equation}
Finally we have a formula for $ B_1(\chi) $. Denote $ \Gamma_\alpha' := \Gamma_\alpha/m $.  Adding \eqref{Equation "Bernoulli intermediate - "} and \eqref{Equation "Bernoulli intermediate + "} we get,
\begin{align}\label{Equation "Temp2"}
	n B_1(\chi) &= \sum_{\alpha\in I_2} \alpha \Gamma_\alpha'\\ \nonumber
	&= \sum_{\alpha\in I_2} \alpha \sum_{k=0}^{n-1} k c_k(\alpha)\\ \label{Equation "Bernoulli general"}
	&= \sum_{\alpha\in I_2} \alpha \sum_{k=0}^{n-1} k \sum_{a=1}^{n} \chi_1(a+km)\beta(a,\alpha).
\end{align}

This completes the proof
\end{proof}

\begin{proof}[Proof of Theorem \ref{Theorem "Class Number main"}]
	Now we shall deduce Theorem \ref{Theorem "Class Number main"} from Theorem \ref{Theorem "General character sums"}. First we specialize to the case where $ \chi_1, \chi_2 $ are the Legendre symbols modulo $ n,m $ respectively so that $ \chi $ is the Legendre symbol modulo $ D $ in \eqref{Equation "Bernoulli general"}. It remains to prove that, under the assumptions of Theorem \ref{Theorem "Class Number main"}, $ \beta(a,\pm 1) = \sigma(-a,\pm 1) $ for any $ 0\leq a\leq n-1 $. We shall prove this for $ \sigma(-a,1) $, the other case being similar. From the definition of $ \beta $, for this particular case,
	\[
	\beta(a,1) = \#\left\{1\leq l\leq m\ |\ l\equiv a\mod n, \left(\frac{l}{m}\right) = 1\right\}.
	\]
	Since $ n $ is a primitive root modulo $ m $, we see that there exists a unique integer $ k $ such that $ 1\leq k\leq m-1 $ and $ l = [n^k]_m $. Furthermore, since $ n $ is a quadratic non residue modulo $ m $, $ l $ is a quadratic residue modulo $ m $ if and only if $ k $ is even. If $ l\equiv a\mod (n) $, then $ [n^k]_m\equiv a\mod n $ and from \eqref{Equation "Congruence of ak"}, $ ma_k(m)\equiv -a\mod n $. Consider this map $ l\mapsto k$ under this correspondence. It is easily verified that this is a bijection and therefore $ \beta(a,1) = \sigma(-a,1) $. Similarly one may also prove that $ \beta(a,-1) = \sigma(-a,-1) $. This completes the proof of Theorem \ref{Theorem "Class Number main"}.
\end{proof}

Furthermore, the symmetries satisfied by $ a_k(m) $ can be used to deduce certain relations among $ \sigma(a,1) $ and $ \sigma(n-a,1) $. Understanding these symmetries in the simplest case of $ n=3 $ leads to the proof of the corollaries to Theorem \ref{Theorem "Class Number main"}.

\subsection{Proof of Corollaries \ref{Corollary "Class Number corollary"} and \ref{Corollary "Class Number corollary 2"}}

From now onwards, we shall restrict ourselves to the setting of Theorem \ref{Theorem "Class Number main"}, in particular, we shall assume that $ n=3 $, and make the appropriate assumptions on $ m,\chi_1,\chi_2, \beta $ etc. As observed above, if $ 3 $ is a primitive root modulo $ m $, then $ m\equiv 2\mod 3 $.

Since $ m $ is a prime, we have $ \beta (2,0)=1 $ and $ \beta(0,0) = \beta(1,0)=0 $. This gives us
\[
\beta(a,1) + \beta(a,-1)= \begin{cases}
	\left\lfloor m/3\right\rfloor+1 &\mbox{ if } a = 1\\
	\left\lfloor m/3\right\rfloor &\mbox{ if } a\neq  1.
\end{cases}
\]
Then for $ \alpha\in I_2 $, from \eqref{Equation "ck formula"} we have,
\begin{align*}
	c_0(\alpha) &= \beta(1,\alpha) -\beta(2,\alpha)\\
	c_1(\alpha) &= -\beta(0,\alpha) + \beta(2,\alpha)\\
	c_2(\alpha) &= \beta(0,\alpha) -\beta(1,\alpha).
\end{align*}
Recall that $ (m-1)/2 $ is even as we have assumed $ m\equiv 1\mod 4 $. Rewriting the definition of $ \sigma(0,1) $ for our particular case, we get
\begin{align*}
	\sigma(0,1) &= \#\left\{1\leq k\leq (m-1)\ |\ a_k(m)\equiv 0 \mod 3\mbox{ and } 2\mid k  \right\}\\
	&= \#\left\{1\leq k\leq (m-1)\ |\ a_k(m)= 0\mbox{ and } 2\mid k  \right\},
\end{align*}
because $ 0\leq a_k(m)\leq 2 $. Call the set on the right hand side above as $ \Sigma_0 $. Similarly define
\[
\Sigma_2 := \left\{1\leq k\leq (m-1)\ |\ a_k(m)= 2\mbox{ and } 2\mid k  \right\}
\]
Define the map $ \phi : \Sigma_0 \to \Sigma_2 $ as
\[
\phi(k) := \begin{cases}
k+\frac{m-1}{2}&\mbox{if } k < (m-1)/2\\
k-\frac{m-1}{2}&\mbox{if } k > (m-1)/2.
\end{cases}
\]
From Lemma \ref{Lemma 5} we observe that $ \phi $ is well defined\footnote{We also see that $ (m-1)/2 \notin \Sigma_0 $.} and it is clear that $ \phi $ is a bijection. Now we observe that $ \sigma(1,1) = \# \Sigma_2 $ since $ m\equiv 2\mod 3 $. Hence $ \sigma(0,1) = \sigma(1,1) $. Similarly we can prove that $ \sigma(0,-1) = \sigma(1,-1) $. From the discussions leading to the proof of Theorem \ref{Theorem "Class Number main"} from Theorem \ref{Theorem "General character sums"} above, the above equalities can be rewritten in terms of $ \beta $ as $ \beta(0,\pm 1) = \beta(2, \pm 1) $. In particular $ c(1,1)=c(1,-1)=0 $ and hence
\[
\Gamma_{\pm 1}' = c_1(\pm 1) + 2c_2(\pm 1) = 2c_2(\pm 1).
\]
From \eqref{Equation "Temp2"} we see that 
\[
3B_1(\chi) = \Gamma_1'-\Gamma_{-1}'.
\]
In this case we have,
\begin{equation}\label{Equation "Case 1"}
	3B_1(\chi) = 2c_2(1) - 2c_2(-1).
\end{equation}
Observe that,
\[
\beta(0,1) + \beta(1,1) + \beta(2,1) = \#\left\{1\leq n\leq m\ |\ \chi_2(n)=1\right\} = \frac{m-1}{2}.
\]
Similarly,
\[
\beta(0,-1) + \beta(1,-1) + \beta(2,-1) = \#\left\{1\leq n\leq m\ |\ \chi_2(n)=-1\right\} = \frac{m-1}{2}.
\]
Now, define the following variables as in the two cases,
\begin{align*}
	x_1 &= \beta(0,1) = \beta(2,1) && x_2 = \beta(0,-1) = \beta(2,-1)\\
	x_3 &= \beta(1,1) && x_4 = \beta(1,-1)
\end{align*}
Now define the matrix $ A $ as 
\begin{equation}\label{Equation "Matrix A"}
	A:=\begin{pmatrix}
		1&1&0&0\\
		0&0&1&1\\
		2&0&1&0\\
		0&2&0&1
	\end{pmatrix}.
\end{equation}
Then it is clear that the variable $ x_1,x_2,x_3,x_4 $ satisfy
\begin{equation}\label{Equation "Linear System"}
	A\begin{pmatrix}
		x_1\\x_2\\x_3\\x_4
	\end{pmatrix}
	=
	\begin{pmatrix}
		\left\lfloor\frac{m-1}{3}\right\rfloor\\
		\left\lfloor\frac{m-1}{3}\right\rfloor+1\\
		\frac{m-1}{2}\\
		\frac{m-1}{2}
	\end{pmatrix}.
\end{equation}
The matrix $ A $ has rank equal to $ 3 $ and the kernel is spanned by the vector 
\begin{equation}\label{Equation "Nullity vector"}
	\textbf{v} = \begin{pmatrix}
		1\\-1\\-2\\2
	\end{pmatrix}.
\end{equation}
One particular solution in $ \Q $ of the system \eqref{Equation "Linear System"} is given by 
\[
\begin{pmatrix}
	x_1\\x_2\\x_3\\x_4
\end{pmatrix} = 
\begin{pmatrix}
	0\\
	\left\lfloor\frac{m-1}{3}\right\rfloor\\
	\frac{m-1}{2}\\
	\left\lfloor\frac{m-1}{3}\right\rfloor + 1 - \frac{m-1}{2}
\end{pmatrix}.
\]
Therefore any solution should be of the form 
\[
\begin{pmatrix}
	t\\
	-t+\left\lfloor\frac{m-1}{3}\right\rfloor\\
	-2t+\frac{m-1}{2}\\
	2t + \left\lfloor\frac{m-1}{3}\right\rfloor + 1 - \frac{m-1}{2}
\end{pmatrix},
\]
for some $ t $.

Now from \eqref{Equation "Case 1"} we see that,
\begin{align*}
	3B_1(\chi) &= 2 (x_1-x_3) - 2(x_2-x_4)\\
	&= 12t+2-2(m-1)\\
	&= 12\beta(0,1) + 2 - 2(m-1)\\
	&= 12\sigma(0,1) + 2 - 2(m-1).
\end{align*}
If we suppose that $ m = 3k+2 $ for some integer $ k $, then it follows that 
\[
h(3m) = 2k - 4\sigma(0,1).
\]
This completes the proof of Corollary \ref{Corollary "Class Number corollary"}. Corollary \ref{Corollary "Class Number corollary 2"} clearly follows from here.

\section{Concluding Remarks}\label{Section "Concluding Remarks"}

Theorem \ref{Theorem "Rudnick"} concerns only the classical sum of digits. Many twisted sums such as those considered in the present paper as well as those considered by Girstmair \cite{Girstmair 2} may have non trivial relations in the function field case. This will be worth exploring.

The congruence conditions in Theorems \ref{Theorem "Class Number main"} ($ n\equiv 3 \mod 4 $, $ m\equiv 1 \mod 3 $) \ref{Theorem "Congruence for S"} ($ m\equiv 3 \mod 4 $) is necessary because otherwise many of the sums of digits considered will vanish. It would be interesting to find relations among the digits and class number in this case as well.

The growth of $ \sigma (a,\alpha) $ is of independent interest in itself. This may lead to very nice interpretations for the class number. For example, Corollary \ref{Corollary "Class Number corollary 3"} can be used to interpret the class number as a measure of dependence of the events ``$ l $ is divisible by $ 3 $" and ``$ l $ is a quadratic residue modulo $ m $". Indeed starting from estimates for $ \sigma(a,\alpha) $s, it would be interesting to reprove certain classical results on the size of the class numbers of imaginary quadratic fields.


\begin{thebibliography}{25}
	
	\bibitem{Bha-Mur} A. Bhand, M. R. Murty, Class Numbers of Quadratic Fields, Hardy-Ramanujan J. 42 (2019), 17–25.
	
	\bibitem{Girstmair 1} K. Girstmair, A “popular” class number formula, American Math. Monthly, 101 (1994),  no. 10, 997-1001.
	
	\bibitem{Girstmair 2} K. Girstmair, The digits of $1/p$ in connection with class number factors, Acta Arith., 67
	(1994), no. 4, 381-386.
	
	\bibitem{Girstmair 3} K. Girstmair, Periodische Dezimalbr\"uche - was nicht jeder dar\"uber weiss, Jahrbruch
	\"Uberblicke Mathematik, 1995, 163-179, Vieweg, Braunschweig, 1995.
	
	\bibitem{Heath-Brown} Heath-Brown, Artin's conjecture for primitive roots. Quart. J. Math. Oxford Ser. (2) 37 (1986), no. 145, 27–38.
	
	\bibitem{Hooley} C. Hooley, Artin’s conjecture for primitive roots, J. Reine Angew. Math. 225 (1967), 209–220.
	
	\bibitem{Iwa} K. Iwasawa, Lectures on P-Adic L-Functions. (AM-74), Volume 74, Annals of Mathematics Studies, Princeton University Press.

	\bibitem{Kohnen-Ono} W.Kohnen, K.Ono, Indivisibility of class numbers of imaginary quadratic ﬁelds and orders Tate–Shafarevich groups	of elliptic curves with complex multiplication. Invent. Math. 135, 387–398 (1999)
	
	\bibitem{Moore} P. Moree, Artin's Primitive root conjecture - a survey - (with contributions by A.C. Cojocaru, W. Gajda and H. Graves)
	To the memory of John L. Selfridge (1927-2010) Integers 12 (2012), no. 6, 1305–1416.
	
	\bibitem{Ram-Tha} M. R. Murty, R. Thangadurai, The class number of $ \Q(\sqrt{-p}) $ and digits of $ 1/p $ Proc. Amer. Math. Soc. 139 (2011), no. 4, 1277–1289.
	
	\bibitem{Siegel} C.Siegel, \"Uber die Classenzahl quadratischer Zahlkrper. Acta Arith. (in German) 1, 83–86 (1935)

	\bibitem{Sound - divisibility} K.Soundararajan, Divisibility of class numbers of imaginary quadratic ﬁelds. J. Lond. Math Soc. 61, 681–690 (2000).
	
	\bibitem{Larry W- Cyclotomic fields}  L. Washington, Introduction to Cyclotomic Fields, Springer (1982).

\end{thebibliography}
\end{document}